\documentclass{amsart}

\usepackage{amssymb, amsmath,  amsfonts }

\newtheorem{thm}{Theorem}[section]
\newtheorem{prop}[thm]{Proposition}
\newtheorem{lemma}[thm]{Lemma}
\newtheorem{cor}[thm]{Corollary}
\newtheorem{example}[thm]{Example}
\newtheorem{remark}[thm]{Remark}

\numberwithin{equation}{section}

\begin{document}{\allowdisplaybreaks[4]

\title[Classical aspects of quantum cohomology of   flag varieties]{Classical aspects of quantum cohomology of generalized  flag varieties}

\author{Naichung Conan Leung}
\address{The Institute of Mathematical Sciences and Department of Mathematics,
           The Chinese University of Hong Kong, Shatin, Hong Kong}
\email{leung@math.cuhk.edu.hk}
\thanks{ 
 }

\author{Changzheng Li}
\address{School of Mathematics, Korea Institute for Advanced Study, 87
Hoegiro, Dongdaemun-gu, Seoul, 130-722, Korea}
\email{czli@kias.re.kr}

\date{
      }


\keywords{Gromov-Witten invariants. Quantum cohomology. Flag varieties.}


\begin{abstract}
We show that various genus zero  Gromov-Witten invariants for flag
varieties representing different homology classes are indeed the
same. 
In particular, many of them are classical intersection numbers of
Schubert cycles.
\end{abstract}

\maketitle

\section{Introduction}
A  generalized  flag variety $G/P$ is the quotient of a
simply-connected complex simple Lie group $G$ 
 by a parabolic subgroup $P$ of $G$.
   The (small) quantum cohomology ring $QH^*(G/P)$ of $G/P$      is
    a deformation of the ring structure on  $H^*(G/P)$ by incorporating three-pointed, genus zero Gromov-Witten invariants of $G/P$.
    The presentation of the ring structure on $QH^*(G/P)$ in special cases have been studied by many mathematicians (see e.g. \cite{sietian},
         \cite{kim}, \cite{CFon}, \cite{BKT2}, \cite{chper}, \cite{peterson}, \cite{lamshi}, \cite{fu11} and references therein).
    From the viewpoint  of enumerative geometry, it is desirable to have (positive) combinatorial formulas for these Gromov-Witten invariants.
 Thanks to the Peterson-Woodward comparison formula \cite{wo}, all these Gromov-Witten invariants for $G/P$ can  be recovered from  the
 Gromov-Witten invariants for the special case of a complete flag variety $G/B$, where $B$ is a Borel subgroup. 

In \cite{leungli33}, with the help of the Peterson-Woodward comparison
formula, we established a natural  filtered algebra structure
on $QH^{\ast}\left(  G/B\right)$. In this article, we use the structures of
this filtration to obtain relationships among three-pointed genus zero Gromov-Witten invariants $N_{u, v}^{w, \lambda}$ for $G/B$.
The Gromov-Witten invariants $N_{u, v}^{w, \lambda}$ are  the {structure coefficients}  of the quantum product
     $$\sigma^u\star \sigma^v=\sum_{\lambda \in H_2(G/B, \mathbb{Z}), w}N_{u, v}^{w, \lambda}\mathbf{q}_\lambda\sigma^w$$
  of the Schubert cocycles $\sigma^u$ and $\sigma^v$ in the quantum cohomology $QH^*(G/B)$.
    The evaluation  of $\mathbf{q}$ at the origin gives us the classical intersection product
            $$\sigma^u\cup\sigma^v=\sum_wN_{u, v}^{w, 0}\sigma^w.$$
\noindent Let $\Delta=\{\alpha_1, \cdots,\alpha_n\}$ be a base of simple roots of $G$ and $\{\alpha_1^\vee, \cdots, \alpha_n^\vee\}$ be the simple coroots (see section \ref{prelimiar}
and references therein for more details of the notations).
 The Weyl group $W$ is a Coxeter group generated by simple reflections $\{s_\alpha~|~\alpha\in \Delta\}$. For each $\alpha\in \Delta$, we
 introduce a map $\mbox{sgn}_\alpha: W\rightarrow \{0, 1\}$ defined by $\mbox{sgn}_\alpha(w):=1$ if
  $\ell(w)-\ell(ws_\alpha)>0$, and $0$ otherwise. Here $\ell: W\rightarrow \mathbb{Z}_{\geq 0}$ denotes the   length function.
 Let   $\mathfrak{h}$ be the dual of the vector space $\mathfrak{h}^*:=\oplus_{\alpha\in \Delta}\mathbb{C}\alpha$ and
 $\langle \cdot, \cdot \rangle: \mathfrak{h}^*\times \mathfrak{h}\rightarrow \mathbb{C}$ be the natural pairing.
  Note that $H_2(G/B, \mathbb{Z})$ can be canonically identified with the coroot lattice $Q^\vee:=\oplus_{\alpha\in \Delta}\mathbb{Z}\alpha^\vee \subset \mathfrak{h}$.
 We prove
\bigskip

\begin{thm}\label{mainthm}
   For any $u, v, w\in W$ and for any $\lambda\in Q^\vee$, we have
   \begin{enumerate}
      \item   $N_{u, v}^{w, \lambda}=0$  unless  {\upshape $\mbox{sgn}_\alpha(w)+\langle \alpha, \lambda\rangle \leq \mbox{sgn}_\alpha(u)+\mbox{sgn}_\alpha(v)$} for all $\alpha\in \Delta.$
     \item  Suppose    {\upshape $ \mbox{sgn}_\alpha(w)+\langle \alpha, \lambda\rangle =\mbox{sgn}_\alpha(u)+\mbox{sgn}_\alpha(v)=2$} for some  $\alpha\in \Delta$, then
           {\upshape  $$N_{u, v}^{w, \lambda}=N_{us_\alpha, vs_\alpha}^{w, \lambda-\alpha^\vee}=
                 \begin{cases} N_{u, vs_\alpha}^{ws_\alpha, \lambda-\alpha^\vee}, &   i\!f    \mbox{ sgn}_\alpha(w)=0 \\
                               \vspace{-0.3cm}   & \\
                         N_{u, vs_\alpha}^{ws_\alpha, \lambda}, & i\!f   \mbox{ sgn}_\alpha(w)=1 \,\, {}_{\displaystyle .}  \end{cases}$$
            }
   \end{enumerate}

\end{thm}

We obtain nice applications of the above theorem which demonstrate the
so-called \textquotedblleft\textit{quantum to classical}" principle.

The ``quantum to classical" principle says that certain
three-pointed genus zero Gromov-Witten invariants for a given homogeneous space are classical intersection
numbers for a typically different   homogeneous space. This phenomenon, probably for the first time, occurred in the proof
of quantum Pieri rule for partial flag varieties of type $A$ by
Ciocan-Fontanine \cite{CFon}, and later occurred in the elementary proof of
quantum Pieri rule for complex Grassmannians by Buch \cite{buch} as well as the work \cite{KreschTamvakisLagran}, \cite{KreschTamvakisortho} of Kresch and Tamvakis on Lagrangian and orthogonal Grassmannians. The phrase
``quantum to classical principle" was introduced  by Chaput and Perrin \cite{chper} for the work \cite{BKT1} of  Buch, Kresch and Tamvakis
   on  complex Grassmannians, Lagrangian Grassmannians and
orthogonal Grassmannians, which says that  any three-pointed genus zero
Gromov-Witten invariant on a Grassmannian of aforementioned types is equal to
a classical intersection number on  a partial flag variety of the same Lie
type. Recently,   this principle has been developed by  Buch, Kresch and Tamvakis  for isotropic Grassmannians
of classical types \cite{BKT2}.  For Grassmannians of certain exceptional
types, this principle has also been studied by  Chaput, Manivel and Perrin
(\cite{cmn}, \cite{chper}).  For flag varieties of $A$-type, there are
relevant studies  by Coskun \cite{coskun33}. For the special case of computing the number of lines   in a general complete variety $G/B$, this principle has
also been studied earlier  by the second author and Mihalcea \cite{liMihal}.
In addition, we note that this principle for certain $K$-theoretic
Gromov-Witten invariants
have been studied by Buch-Mihalcea (\cite{buchMihalcea},
\cite{buchMihalcea22}) and by Li-Mihalcea \cite{liMihal}.

Using Theorem \ref{mainthm}, we not only recover most of the above results on
the ``quantum to classical" principle,  but also get new and interesting
results. Further applications for general Lie types other than type $A$  will
be treated in  \cite{leungli44}, \cite{leungli55}.  For
instance in \cite{leungli44}, we could see the applications of Theorem \ref{mainthm} in seeking  quantum Pieri rules with respect
to Chern classes of the dual of the tautological subbundles for Grassmannians
of classical types,  which are not covered in \cite{BKT2} in general.

For type $A_{n}$ case, we note that the Weyl group $W$ is canonically
isomorphic to the permutation group $S_{n+1}$  and $G/B=F\ell_{n+1}%
=\{V_{1}\leqslant\cdots\leqslant V_{n}\leqslant\mathbb{C}^{n+1}~|~
\dim_{\mathbb{C}}V_{j}=j, j=1, \cdots, n\}$. An element $u\in W=S_{n+1}$ is
called a \textit{Grassmannian permutation}, if there exists $1\leq k\leq n$ such
that  $\sigma^{u}\in H^{*}(F\ell_{n+1})$ comes from the pull-back  $\pi^{*}:
H^{*}(Gr(k, n+1))\rightarrow H^{*}(F\ell_{n+1})$ induced from the natural
projection map  $\pi: F\ell_{n+1}\rightarrow\{V_{k}\leqslant\mathbb{C}%
^{n+1}~|~ \dim V_{k}=k\}=Gr(k, n+1)$. Equivalently, a Grassmannian permutation $u\in
W$ is an element such that all the reduced expressions $u=s_{i_{1}}s_{i_{2}%
}\cdots s_{i_{m}}$, where $m=\ell(u)$,  end with the same simple reflection
$s_{k}$. When written in ``one-line" notation, Grassmannian permutations are precisely the permutations with a single descent (see Remark \ref{remdescuvw} for more details).  As an application of Theorem \ref{mainthm}, we have

\begin{thm}\label{mainapplication}
   Let     $u, v, w\in S_{n+1}$ and   $\lambda \in Q^\vee$. If $u$ is of Grassmannian type,  then there exist $v', w'\in S_{n+1}$ such that
     $$N_{u, v}^{w, \lambda}=N_{u, v'}^{w',\, 0}.$$
\end{thm}

\noindent It is also interesting to investigate the above theorem from the point of view of symmetries  on
$QH^{*}(F\ell_{n+1})$, analogous with   the cyclic symmetries shown by
Postnikov \cite{postnikov}.  In
addition, the proof of Theorem \ref{mainapplication} will also describe how to find
$v^{\prime}$ and $w^{\prime}$ (easily). Special cases of Theorem
\ref{mainapplication} enable us to recover the quantum Pieri rule for partial
flag varieties of type $A$ as in \cite{CFon} and the ``quantum to classical"
principle for complex Grassmannians as in \cite{BKT1}.

Geometrically, the Gromov-Witten invariants $N_{u, v}^{w, \lambda}$ count the
number of stable holomorphic maps from the projective line $\mathbb{P}^{1}$,
or more generally a rational curve,  to $G/B$. In particular, they are all
non-negative. There have been closed formulas/algorithms on the classical
intersection product  by Kostant-Kumar \cite{koku} and Duan \cite{duan} and on
the quantum product by the authors \cite{leungli22}. Yet sign cancelations are
involved in all these formulas/algorithms. For a complete flag variety (of
general type), the problem of finding a positive formula on either side
remains open. The ``quantum to classical" principle could help us to reduce such
an issue on the quantum side to the classical side. When $G=SL(n+1,
\mathbb{C})$, we note that positive formulas on the classical intersection
numbers have been given by Coskun \cite{coskun11}, \cite{coskun22}.

 The proof of Theorem \ref{mainthm} uses functorial relationships established
by the authors in \cite{leungli33} and it is combinatorial in nature. However,
for a special case (of $\lambda=\alpha^{\vee}$) of Theorem \ref{mainthm},
both a geometric proof of it and a combinatorial proof of its equivariant
extension can be found in \cite{liMihal}.  We also wish to see a geometric
proof of this theorem in the future.

\section{Proofs of theorems}
In this section, we first   fix the notations in section 2.1. Then we prove our  first  main theorem    in section 2.2.
  Finally in section 2.3, we obtain our second main theorem, as an  application of the first main theorem.

\subsection{Notations}\label{prelimiar}

  More details on Lie theory can be found for example in \cite{hum}, \cite{humalg}.

Let $G$ be a simply-connected complex simple Lie group of rank $n$
and
   $B\subset G$ be a Borel subgroup.  Let  $\Delta=\{\alpha_1, \cdots, \alpha_n\}\subset \mathfrak{h}^*$ be the simple
   roots and  $\{\alpha_1^\vee, \cdots, \alpha_n^\vee\}\subset\mathfrak{h}$ be the simple
   coroots, where
          $\mathfrak{h}$ is the corresponding Cartan subalgebra of $(G, B)$.
      Let $Q^\vee=\bigoplus_{i=1}^n\mathbb{Z}\alpha_i^\vee$  and  $\rho=\sum_{i=1}^n\chi_i\in \mathfrak{h}^*$.
   Here $\chi_i$'s are the fundamental weights, which for any $i, j$ satisfy
         $\langle \chi_i, \alpha_j^\vee\rangle =\delta_{i, j}$
       with respect to the natural pairing  $\langle\cdot, \cdot\rangle
                    :\mathfrak{h}^*\times\mathfrak{h}\rightarrow \mathbb{C}$.
   The Weyl group $W$  is generated by
           $\{s_1, \cdots, s_n\}$, where each $s_i=s_{\alpha_i}$ is a simple reflection   on   $\mathfrak{h}^*$
         defined  by    $
                s_i(\beta)=\beta-\langle \beta, \alpha_i^\vee\rangle\alpha_i.$
                            The root system is given by $R=W\cdot\Delta=R^+\sqcup (-R^+)$, where
                       $R^+=R\cap \bigoplus_{i=1}^n{\mathbb{Z}_{\geq 0}}\alpha_i$ is the set of positive roots.
        Each parabolic subgroup $P\supset B$ is in one-to-one correspondence with
         a subset $\Delta_P\subset \Delta$.  In fact, $\Delta_P$ is the set of simple roots of a Levi subgroup of $P$.
     Let $\ell: W\rightarrow \mathbb{Z}_{\geq 0}$ be the length function,   $W_P$ denote
      the Weyl subgroup generated by $\{s_\alpha~|~ \alpha\in
      \Delta_P\}$ and
        $W^{P}$ denote the subset $\{w\in W |  \ell(w)\leq \ell(v),\, \forall v\in wW_{P}\}$.
      Each coset in $W/W_{P}$ has a unique minimal length representative in $W^P$.

   The (co)homology of a (generalized) flag variety $X=G/P$ has an additive basis of Schubert (co)homology classes indexed by $W^P$:
       $ H_*(X,\mathbb{Z})=\bigoplus_{v\in W^P}\mathbb{Z}\sigma_{v}$, \,\, $H^*(X,\mathbb{Z})=\bigoplus_{u\in W^P}\mathbb{Z}\sigma^{u}$ with
        $   \langle\sigma^{u}, \sigma_v\rangle =\delta_{u, v}$ for any $u, v\in W^P$ \cite{bgg}.
        Note that each   $\sigma_u$ (resp. $\sigma^u$) is a class in the $2\ell(u)^{\mbox{th}}$-(co)homology. In particular,
          $H_2(X,\mathbb{Z})=\bigoplus_{\alpha_i\in \Delta\setminus\Delta_P}
         \mathbb{Z}\sigma_{s_i}$ can be canonically identified with
            $Q^\vee/Q^\vee_P$, where
          $Q^\vee_P:=\bigoplus_{\alpha_i\in \Delta_P}\mathbb{Z}\alpha_i^\vee$.
               For each $\alpha_j\in \Delta\setminus\Delta_P$, we introduce a formal variable $q_{\alpha_j^\vee+Q^\vee_P}$.
          For  $\lambda_P=\sum_{\alpha_j\in \Delta\setminus\Delta_P}a_j\alpha_j^\vee+Q^\vee_P\in H_2(X, \mathbb{Z})$,
          we denote $q_{\lambda_P}=\prod_{\alpha_j\in \Delta\setminus\Delta_P}q_{\alpha_j^\vee+Q^\vee_P}^{a_j}$.
  The      (\textbf{small}) \textbf{quantum cohomology}    $QH^*(X)=(H^*(X)\otimes\mathbb{Q}[\mathbf{q}],  \star)$ of $X$
     is   a commutative ring and has a $\mathbb{Q}[\mathbf{q}]$-basis of Schubert classes $\sigma^u=\sigma^u\otimes 1$.
       The \textit{quantum Schubert  structure constants} $N_{u, v}^{w, \lambda_P}$ for the
       quantum product
                 $$\sigma^u\star \sigma^v =\sum_{w\in W^P, \lambda_P\in Q^\vee/Q^\vee_P}    N_{u,v}^{w, \lambda_P}q_{\lambda_P}\sigma^w$$
  are
   genus zero Gromov-Witten invariants given by
       $N_{u, v}^{w, \lambda_P}=\int_{\overline{\mathcal{M}}_{0, 3}(X, \lambda_P)}\mbox{ev}_1^* (\sigma^u)\cup\mbox{ev}_2^*(\sigma^v)
                            \cup \mbox{ev}_3^*((\sigma^w)^\sharp)$.
  Here  $\overline{\mathcal{M}}_{0, 3}(X, \lambda_P)$ is the moduli space of stable maps of degree $\lambda_P\in H_2(X, \mathbb{Z})$ of $3$-pointed genus zero curves into $X$,
  $\mbox{ev}_i: \overline{\mathcal{M}}_{0, 3}(X, \lambda_P)\to X$  is   the $i$-th canonical evaluation map and
    $\{(\sigma^{w})^\sharp~|~ w\in W^P\}$ are the elements in $H^*(X)$ satisfying $\int_X(\sigma^{w'})^\sharp \cup\sigma^{w''}=\delta_{w', w''}$ for any $w', w''\in W^P$ \cite{fupa}.  Note that $N_{u, v}^{w, \lambda_P}=0$ unless $q_{\lambda_P}\in \mathbb{Q}[\mathbf{q}]$.
  It is a well-known fact that these Gromov-Witten invariants $N_{u, v}^{w, \lambda_P}$ of the flag variety $X$ are enumerative, counting  the number of certain
   holomorphic maps from $\mathbb{P}^1$ to $X$. In particular,  they are all   \textbf{non-negative} integers. (Thus for the special case of a flag variety $X$, we can also define $QH^*(X)$ over $\mathbb{Z}$ whenever we wish.)

In analog with the classical cohomology, there is a natural $\mathbb{Z}$-grading on the quantum cohomolgy $QH^*(X)$, making it a $\mathbb{Z}$-graded ring:
       $$QH^*(X)=\bigoplus_{n\in \mathbb{Z}}\Big(\bigoplus_{\deg(q_{\lambda_P}\sigma^w)=n}\mathbb{Q}q_{\lambda_P}\sigma^w\Big).\qquad\qquad (*)$$
  Here the degree of $q_{\lambda_P}\sigma^w$, where  $\lambda_P=\sum_{\alpha_j\in \Delta\setminus\Delta_P}a_j\alpha_j^\vee+Q^\vee_P\in H_2(X, \mathbb{Z})$, is given by
     $$ \deg(q_{\lambda_P}\sigma^w)=\ell(w)+\sum_{\alpha_j\in \Delta\setminus\Delta_P}a_j\langle\sigma_{s_j}, c_1(X)\rangle,$$
  in which $\langle \cdot, \cdot\rangle$ is the natural pairing between homology and cohomology classes, and an explicit description of the first Chern class $c_1(X)$ can be found for example in \cite{fw}.
 When $P=B$, we have $\Delta_P=\emptyset, Q^\vee_P=0$, $W_P=\{1\}$ and $W^P=W$. In this case, we simply denote $\lambda=\lambda_P$ and $q_j=q_{\alpha_j^\vee}$.
As a direct consequence of the $\mathbb{Z}$-graded ring structure $(*)$ of $QH^*(G/B)$,
   for any $u, v, w\in W$ and for any $\lambda\in Q^\vee$, we have
     $$N_{u, v}^{w, \lambda}=0  \mbox{ unless } \ell(w)+\langle 2\rho, \lambda\rangle=\ell(u)+\ell(v).$$

\subsection{Proof of Theorem \ref{mainthm}}
This subsection is devoted to the proof of Theorem \ref{mainthm}. The main arguments are given in section \ref{subarguproofmainthm},
  based on the results in \cite{leungli33} which will be  reviewed in section \ref{subZ2filtration}. We will introduce the
Peterson-Woodward comparison formula first in section \ref{subPWcomparison}. This comparison formula not only plays an important role in
obtaining the results in \cite{leungli33}, but  also shows us that it suffices to know
 all quantum Schubert structure constants  $N_{u, v}^{w, \lambda}$ for $G/B$ in order to know all quantum Schubert structure constants  for all $G/P$'s.

\subsubsection{Peterson-Woodward comparison formula}\label{subPWcomparison}
   We use $\star_P$ to distinguish the quantum products for different  flag varieties $G/P$'s (when needed). For any $u, v\in W^P$,
   we have
      $\sigma^u\star_P\sigma^v=\sum\limits_{w\in W^P, \lambda_P\in Q^\vee/Q^\vee_P}    N_{u,v}^{w, \lambda_P}q_{\lambda_P}\sigma^w$.
 Note $W^P\subset W$.  The classes
  $\sigma^u$ and $\sigma^v$ in $QH^*(G/P)$
   can both be treated as classes in $QH^*(G/B)$ naturally.
   Whenever  referring to $N_{u, v}^{w, \lambda}$ where
     $\lambda\in Q^\vee$,   we are considering   the
   quantum product in $QH^*(G/B)$:  $\sigma^u\star_B\sigma^v=\sum\limits_{w\in W, \lambda\in Q^\vee}    N_{u,v}^{w, \lambda}q_{\lambda}\sigma^w$.

      \begin{prop}[Peterson-Woodward comparison formula  \cite{wo}; see also \cite{lamshi}]\label{comparison}
    ${}$
    \begin{enumerate}
      \item Let $\lambda_P\in Q^\vee/Q_P^\vee$. Then there is a unique $\lambda_B\in Q^\vee$ such that $\lambda_P=\lambda_B+Q_P^\vee$ and
                $\langle \gamma, \lambda_B\rangle  \in \{0, -1\}$ for all $\gamma\in R^+_P \,\,(=R^+\cap \bigoplus_{\beta\in \Delta_P}\mathbb{Z}\beta)$.
      \item              For every $u, v, w\in W^P$, we have  $$N_{u,v}^{w, \lambda_P }=N_{u, v}^{w\omega_P\omega_{P'},   \lambda_B}.$$
               Here  $\omega_{P}$ (resp. $\omega_{P'}$) is the longest element
                in the Weyl subgroup $W_P$ (resp. $W_{P'}$), where    $\Delta_{P'}=\{\beta\in \Delta_P~|~\langle  \beta, \lambda_B\rangle =0\}$.
      \end{enumerate}
    \end{prop}

Thanks to the above proposition, we obtain an injection  of vector spaces
      $$\psi_{\Delta, \Delta_P}: QH^*(G/P)\longrightarrow QH^*(G/B) \,\,\mbox{defined by }  q_{\lambda_P}\sigma^w\mapsto q_{\lambda_B}\sigma^{w\omega_{P}\omega_{P'}}.$$
For the special case of a singleton subset $\{\alpha\}\subset \Delta$, we denote $P_\alpha=P$ and simply denote $\psi_\alpha=\psi_{\Delta, \{\alpha\}}$. In this case,
we note that $R_{P_\alpha}^+=\{\alpha\}, Q_{P_\alpha}^\vee=\mathbb{Z}\alpha^\vee$, and we have the natural fibration
 $P_\alpha/B\rightarrow G/B\rightarrow G/P_\alpha  \mbox{ with } P_\alpha/B\cong \mathbb{P}^1.$
\begin{example}
   Let $\lambda_{P_\alpha}=\beta^\vee+Q_{P_\alpha}^\vee$ where $\beta\in \Delta\setminus\{\alpha\}$.  Then we have $\langle \alpha, \beta^\vee\rangle \in \{0, -1, -2, -3\}$.
   Furthermore, we have

      $\hspace{2.7cm}\psi_\alpha(q_{\lambda_{P_\alpha}})=\begin{cases}
          q_{\beta^\vee}, &\mbox{ if } \langle \alpha, \beta^\vee\rangle =0\\
           s_\alpha q_{\beta^\vee}, &\mbox{ if } \langle \alpha, \beta^\vee\rangle =-1\\
            q_{\alpha^\vee}q_{\beta^\vee}, &\mbox{ if } \langle \alpha, \beta^\vee\rangle =-2\\
           s_\alpha  q_{\alpha^\vee}q_{\beta^\vee}, &\mbox{ if } \langle \alpha, \beta^\vee\rangle =-3 \,\, {\displaystyle .}
      \end{cases}$

\noindent More generally, 
 we consider  $\lambda_{P_\alpha}\!\!=\lambda'+Q_{P_\alpha}^\vee\!\!\in Q^\vee/Q_{P_\alpha}^\vee$, where $\lambda'=\sum_{\beta\in \Delta\setminus\{\alpha\}}c_\beta\beta^\vee\in Q^\vee.$
  Setting $m=\langle \alpha, \lambda'\rangle$, we have
    $\psi_\alpha(q_{\lambda_{P_\alpha}}\sigma^w)=\begin{cases}
          q_{\lambda'-{m\over 2}\alpha^\vee}\sigma^w, &\mbox{ if } m \mbox{ is even}\\
           q_{\lambda'-{m+1\over 2}\alpha^\vee}\sigma^{ws_\alpha}, &\mbox{ if } m \mbox{ is odd}\\
      \end{cases}$.
\end{example}

\subsubsection{$\mathbb{Z}^2$-filtrations on $QH^*(G/B)$}\label{subZ2filtration}

As shown in \cite{leungli33}, given any parabolic subgroup $P$ of $G$ containing $B$,
 we can construct a $\mathbb{Z}^{|\Delta_P|+1}$-filtration on $QH^*(G/B)$.
In this subsection, we review the main results in \cite{leungli33} for the special case of a parabolic subgroup that corresponds to a singleton subset  $\{\alpha\}$. Using them,
we   prove Theorem \ref{mainthm} in next subsection.

 Recall that a natural basis of $QH^*(G/B)[q_1^{-1}, \cdots, q_n^{-1}]$ is given by  $q_\lambda\sigma^{w}$'s labelled by $(w, \lambda)\in W\times Q^\vee$.     Note that
   $q_\lambda\sigma^{w}\in QH^*(G/B)$ if and only if $q_\lambda\in \mathbb{Q}[\mathbf{q}]$ is a polynomial.
  In order to obtain a  filtration on $QH^*(G/B)$, we just need to define (nice) gradings for a given  basis of it.
  Furthermore as in the introduction,  we have defined a map $\mbox{sgn}_\alpha$ with respect to  any given simple root $\alpha\in \Delta$  as follows.
         $$\mbox{sgn}_\alpha:  W\rightarrow \{0, 1\};\,\, \mbox{sgn}_\alpha(w)=\begin{cases}
             1, & \mbox{if } \ell(w)-\ell(ws_\alpha)>0\\
             0, &\mbox{if } \ell(w)-\ell(ws_\alpha)\leq 0
         \end{cases}.
$$
\noindent Note    that
              $\ell(w)-\ell(ws_\alpha)=\pm 1$ and that $\ell(w)-\ell(ws_\alpha)=1$ if and only if $w(\alpha)\in -R^+$, which holds
              if and only if $w=us_\alpha \mbox{ for a unique } u\in W^{P_\alpha}$. (See e.g. \cite{humr}.)
We can define a grading map $gr_\alpha$ with
   respect to a given simple root $\alpha\in \Delta$ as follows.
     \begin{align*}
         gr_\alpha: & W\times Q^\vee\longrightarrow \mathbb{Z}^2;\\
                  &\,  gr_\alpha(q_\lambda\sigma^{w})=(\mbox{sgn}_\alpha(w)+\langle \alpha, \lambda\rangle, \ell(w)+\langle 2\rho, \lambda\rangle-\mbox{sgn}_\alpha(w)-\langle \alpha, \lambda\rangle).
     \end{align*}
Here we are using  \textit{lexicographical
order}    on    $\mathbb{Z}^2$. That is,
    $\mathbf{a}=(a_1, a_2)<\mathbf{b}=(b_1, b_2)$ if and  only if  either   ($a_1=b_1$ and $a_2<b_2$) or     $a_{1}<b_{1}$ holds.
     The above $\mathbb{Z}^2$-grading   of $q_\lambda\sigma^{w}$ can recover the degree grading of it as in section \ref{prelimiar}. Precisely, if we
     write $gr_\alpha(q_\lambda\sigma^{w})=(i, j)$, then $\deg(q_\lambda\sigma^{w})=i+j$.

\begin{remark}
   Following from Corollary 3.13 of \cite{leungli33}, our grading map $gr_\alpha$ coincides with the grading map in Definition 2.8 of \cite{leungli33} by using
   the Peterson-Woodward lifting map $\psi_\alpha=\psi_{\Delta, \{\alpha\}}$.

\end{remark}

  As a consequence, we obtain a     family    $\mathcal{F}=\{F_{\mathbf{a}}\}_{\mathbf{a}\in \mathbb{Z}^2}$ of vector subspaces of $QH^*(G/B)$,
   where   $F_{\mathbf{a}}:=\bigoplus\limits_{gr_\alpha(q_\lambda\sigma^{w})\leq \mathbf{a}}\mathbb{Q}q_\lambda\sigma^{w}\subset QH^*(G/B)$,
  and
  the associated graded vector space  $Gr^{\mathcal{F}}(QH^*(G/B))=\bigoplus_{\mathbf{a}\in \mathbb{Z}^2} Gr_\mathbf{a}^{\mathcal{F}}$
  with respect to $\mathcal{F}$, where $Gr_{\mathbf{a}}^{\mathcal{F}}:=F_{\mathbf{a}}\big/\cup_{\mathbf{b}<\mathbf{a}}F_{\mathbf{b}}.$

\begin{prop}[Theorem 1.2 of \cite{leungli33}]\label{propfilteralgebra}
      $QH^*(G/B)$ is  a   $\mathbb{Z}^2$-filtered algebra with respect to   $\mathcal{F}$. That is, we have $F_\mathbf{a}\star F_\mathbf{b}\subset F_{\mathbf{a}+\mathbf{b}}$
 for any $\mathbf{a}, \mathbf{b}\in \mathbb{Z}^2$.
\end{prop}

Denote     $Gr_{\scriptsize\mbox{vert}}^{\mathcal{F}}(QH^*(G/B)):=\bigoplus\limits_{i\in \mathbb{Z}}Gr_{(i, 0)}^{\mathcal{F}}$ and
  $Gr_{\scriptsize\mbox{hor}}^{\mathcal{F}}(QH^*(G/B)):=\bigoplus\limits_{j\in \mathbb{Z}}Gr_{(0, j)}^{\mathcal{F}}$.
 We take
    the canonical isomorphism  $QH^*(\mathbb{P}^1)\cong {\mathbb{Q}[x, t]\over \langle x^{2}-t\rangle }$ for the fiber of the fibration
    $\mathbb{P}^1\rightarrow G/B\rightarrow G/P_{\alpha}$.

\begin{prop}[Theorem 1.4 of \cite{leungli33}]\label{propfiberisomor}
  The following  maps {\upshape $\Psi_{\scriptsize\mbox{vert}}^\alpha$} and  {\upshape $\Psi_{\scriptsize\mbox{hor}}^\alpha$} are well-defined and they are algebra
    isomorphisms\footnote{In terms of notations in \cite{leungli33},  $\Psi_{\scriptsize\mbox{vert}}^\alpha=\Psi_1$ and  $\Psi_{\scriptsize\mbox{hor}}^\alpha=\Psi_2$.}.
  {\upshape $$\begin{array}{rrc}
       \Psi_{\scriptsize\mbox{vert}}^\alpha:&  QH^*(\mathbb{P}^1) \longrightarrow  Gr_{\scriptsize\mbox{vert}}^{\mathcal{F}}(QH^*(G/B)); &
              \quad x\mapsto            \overline{s_\alpha},\,\,  t\mapsto \overline{q_{\alpha^\vee}}\,\,\,.\\
         \Psi_{\scriptsize\mbox{hor}}^\alpha:&  QH^*(G/P_\alpha) \longrightarrow  Gr_{\scriptsize\mbox{hor}}^{\mathcal{F}}(QH^*(G/B)); &
                  \,\,     q_{\lambda_{P_\alpha}}\sigma^w\mapsto           \overline{\psi_{\alpha}(q_{\lambda_{P_\alpha}}\sigma^w)}\,\,\,.
\end{array}$$
  }
\end{prop}

Here we note that $\overline{s_\alpha} \in Gr^\mathcal{F}_{(1, 0)}\subset Gr_{\scriptsize\mbox{vert}}^\mathcal{F}(QH^*(G/B))$ denotes the
 graded component of $\sigma^{s_\alpha}+\cup_{\mathbf{b}<(1, 0)}F_{\mathbf{b}}$. 
 Similar notations are taken whenever ``$\overline{(\,)}$" is used.
In addition, we have

\begin{prop}[Proposition 3.23 of \cite{leungli33}]\label{propgradingcomp}
   For any  $u\in W^{P_\alpha}$, we have  $\sigma^u\star \sigma^{s_\alpha}=\sigma^{us_\alpha}+\sum_{w, \lambda}b_{w, \lambda}q_\lambda \sigma^w$
    with $gr_\alpha(q_\lambda\sigma^{w})<gr_\alpha(\sigma^{us_\alpha})$ whenever $b_{w, \lambda}\neq 0$.
\end{prop}

The next lemma follows directly from the definition of  the grading map $gr_\alpha$.

\begin{lemma}\label{lemmagradcomp}
  Let $u, v, w\in W$ and   $\lambda \in Q^\vee$. Then
  $gr_\alpha(\sigma^u)+gr_\alpha(\sigma^v)=gr_\alpha(q_\lambda\sigma^{w})$ if and only if both
  $\ell(w)+\langle 2\rho, \lambda\rangle=\ell(u)+\ell(v)$ and {\upshape $\mbox{sgn}_\alpha(w)+\langle \alpha, \lambda\rangle =\mbox{sgn}_\alpha(u)+\mbox{sgn}_\alpha(v)$} hold.

\end{lemma}

\subsubsection{Proof of Theorem \ref{mainthm}}\label{subarguproofmainthm}
The first half of Theorem \ref{mainthm} is a direct consequence of Proposition \ref{propfilteralgebra}. Indeed,
 if
 $\mbox{sgn}_\beta(w)+\langle \beta, \lambda\rangle > \mbox{sgn}_\beta(u)+\mbox{sgn}_\beta(v)$  for some $\beta\in \Delta$,
 then $gr_\beta(q_\lambda  \sigma^w)> gr_\beta( \sigma^u)+gr_\beta( \sigma^v)$.
Since $\sigma^u\star \sigma^v\in F_{gr_\beta(\sigma^u)}\star F_{gr_\beta(\sigma^v)}\subset F_{gr_\beta(\sigma^u)+gr_\beta(\sigma^v)}$, we conclude
 $N_{u, v}^{w, \lambda}=0$.

 It remains to show the second half of Theorem \ref{mainthm}. Note that $\mbox{sgn}_\alpha$ is a map from $W$ to $\{0, 1\}$.
  Since $\mbox{sgn}_\alpha(u)+\mbox{sgn}_\alpha(v)=2$, we have
  $\mbox{sgn}_\alpha(u)=\mbox{sgn}_\alpha(v)=1$. Consequently, $u':=us_\alpha$ and $v':=vs_\alpha$ are both elements in $W^{P_\alpha}$.
In the rest, we can assume   $\ell(w)+\langle 2\rho, \lambda\rangle= \ell(u)+\ell(v)$. (Otherwise,  both $N_{u, v}^{w, \lambda}$ and  $N_{u', v'}^{w, \lambda-\alpha^\vee}$ would vanish,   directly   following   from the standard $\mathbb{Z}$-graded ring structure $(*)$ of $QH^*(G/B)$.)

By Proposition \ref{propgradingcomp}, we have
 $$\overline{\sigma^{u'}\star \sigma^{s_\alpha}}=\overline{ \sigma^u}\in Gr^\mathcal{F}_{gr_\alpha(\sigma^u)}
     \mbox{ and  } \overline{\sigma^{v'}\star \sigma^{s_\alpha}}=\overline{ \sigma^v}\in Gr^\mathcal{F}_{gr_\alpha(\sigma^v)}.$$
  Note that  $QH^*(G/B)$ is an associative and commutative $\mathbb{Z}^2$-filtered algebra with respect to $\mathcal{F}$.
As a consequence,   $Gr^\mathcal{F}(QH^*(G/B))$ is an   associative and commutative $\mathbb{Z}^2$-graded algebra.
Thus we have
 $$\mbox{LHS}:=\overline{\sigma^{u'}\star \sigma^{s_\alpha}}\star \, \overline{\sigma^{v'}\star \sigma^{s_\alpha}} =\overline{ \sigma^u}\star
  \overline{ \sigma^v}=:\mbox{RHS}$$ in $Gr^\mathcal{F}_{gr_\alpha(\sigma^u)+gr_\alpha(\sigma^v)}$.
By Lemma \ref{lemmagradcomp}, we have
$$\mbox{RHS}= \overline{ \sigma^u \star   \sigma^v}= \overline{\sum N_{u, v}^{\tilde w, \tilde \lambda}q_{\tilde \lambda}\sigma^{\tilde w}}= \sum N_{u, v}^{\tilde w, \tilde \lambda}\overline{q_{\tilde \lambda}\sigma^{\tilde w}},$$
 where the summation is over those $(\tilde w, \tilde \lambda)\in W\times Q^\vee$ satisfying
  $\ell(\tilde w)+\langle 2\rho, \tilde \lambda\rangle=\ell(u)+\ell(v)$ and   $\mbox{sgn}_\alpha(\tilde w)+\langle \alpha, \tilde \lambda\rangle =2$.
Let $\star_\alpha$ denote the quantum product for $QH^*(G/P_\alpha)$.  By Proposition \ref{propfiberisomor}, we have
 $$ \mbox{LHS} = (\overline{\sigma^{u'}}\star  \overline{\sigma^{v'}})\star ( \overline{ \sigma^{s_\alpha}}\star  \overline{\sigma^{s_\alpha}})
               =  \Psi_{\scriptsize\mbox{hor}}^\alpha(\sigma^{u'}\star_\alpha\sigma^{v'})\star\, \overline{ q_{\alpha^\vee}}\\
              =\!\sum N_{u', v'}^{w', \lambda_{P_\alpha}}\overline{\psi_{\alpha}(\sigma^{w'}q_{\lambda_{P_\alpha}}) q_{\alpha^\vee}},
$$
the summation over those $(w', \lambda_{P_\alpha})\in W^P\times Q^\vee/Q^\vee_{P_\alpha}$ (with $\lambda_{P_\alpha}$ being effective).
Then we conclude $N_{u, v}^{w, \lambda}=N_{us_\alpha, vs_\alpha}^{w, \lambda-\alpha^\vee}$ by comparing coefficients of both sides.
 Indeed, for $\lambda_{P_\alpha}:=\lambda+Q^\vee_{P_\alpha}$, we have $\lambda_B=\lambda-\alpha^\vee$   via Peterson-Woodward comparison formula
 (by noting  $\langle \alpha, \lambda-\alpha^\vee\rangle=-\mbox{sgn}_\alpha(w)\in\{0, -1\}$).
 Set $w':=w$ if $\mbox{sgn}_\alpha(w)=0$, or $ws_\alpha$ if $\mbox{sgn}_\alpha(w)=1$.
 Note that     $\psi_{\alpha}(\sigma^{w'}q_{\lambda_{P_\alpha}})q_{\alpha^\vee}=\sigma^wq_{\lambda}$.
    We conclude
     $$N_{u, v}^{w, \lambda}=N_{u', v'}^{w', \lambda_{P_\alpha}}=N_{u',v'}^{w, \lambda-\alpha^\vee} .$$

Note that for any $\hat w\in W$, we have
         $$\overline{\sigma^{\hat w}}\star \, \overline{\sigma^{s_\alpha}}=\begin{cases}
              \overline{\sigma^{\hat ws_\alpha}},&\mbox{ if } \mbox{sgn}_\alpha(\hat w)=0\\
              \overline{\sigma^{\hat ws_\alpha}}\star\overline{\sigma^{s_\alpha}} \star \overline{\sigma^{s_\alpha}}
               =  \overline{\sigma^{\hat ws_\alpha}q_{\alpha^\vee}},&\mbox{ if } \mbox{sgn}_\alpha(\hat w)=1
         \end{cases}.$$  Hence, we have
   \begin{align*}
       \overline{ \sigma^u}\star\overline{ \sigma^v} =   \overline{\sigma^{u}}\star \, \overline{\sigma^{v'}\star \sigma^{s_\alpha}}
                                &=\overline{\sigma^{u} \star\, \sigma^{v'}}\star \overline{\sigma^{s_\alpha}}\\
                 &=\overline{\sum N_{u, v'}^{\hat w, \hat \lambda}q_{\hat  \lambda}\sigma^{\hat w}}\star \overline{\sigma^{s_\alpha}}\\
                 &=\sum N_{u, v'}^{\hat w, \hat \lambda}\overline{q_{\hat  \lambda}\sigma^{\hat w s_\alpha}}+
                   \sum N_{u, v'}^{\hat w, \hat \lambda}\overline{q_{\hat  \lambda +\alpha^\vee}\sigma^{\hat w s_\alpha}},
   \end{align*}
 the former (resp. latter) summation over those $(\hat w, \hat \lambda)\in W\times Q^\vee$ satisfying
$\ell(\hat w)+\langle 2\rho, \hat\lambda\rangle=\ell(u)+\ell(v')$,   $\mbox{sgn}_\alpha(\hat w)+\langle \alpha, \hat \lambda\rangle =1$
  and $\mbox{sgn}_\alpha(\hat w)=0$ (resp. $1$).
 Hence, if $\mbox{sgn}_\alpha(w)=0$ (resp. 1), then
   we have $N_{u, v}^{w, \lambda}q_{\lambda} \sigma^{w}=N_{u, v'}^{\hat w, \hat \lambda}  q_{\hat  \lambda +\alpha^\vee}\sigma^{\hat w s_\alpha}
      $ (resp. $N_{u, v'}^{\hat w, \hat \lambda}  q_{\hat  \lambda  }\sigma^{\hat w s_\alpha}$)
   for a unique $(\hat w, \hat \lambda)$ in the latter (resp. former) summation.
  Thus we conclude that
   $N_{u, v}^{w, \lambda}$ equals  $N_{u, vs_\alpha}^{ws_\alpha, \lambda-\alpha^\vee}$ if    $\mbox{sgn}_\alpha(w)=0$,
        or $N_{u, vs_\alpha}^{ws_\alpha, \lambda}$  if   $\mbox{sgn}_\alpha(w)=1  .$

\subsection{Applications}\label{subapplication}
  In this subsection, we   give   applications of Theorem \ref{mainthm} for $\Delta$ of $A$-type case.
 (See the introduction for possible further applications for other cases.) 
 For convenience, we assume the Dynkin diagram of $\Delta$ is given by
 \begin{tabular}{l} \raisebox{-0.4ex}[0pt]{$ \circ\;\!\!\!\!-\!\!\!-\!\!\!-\!\!\!\!\;\circ \cdots \circ\;\!\!\!\!-\!\!\!-\!\!\!-\!\!\! \;  \circ$}\\
                 \raisebox{0.60ex}[0pt]{${\scriptstyle{\alpha_1}\hspace{0.3cm}\alpha_2\hspace{1.0cm}\alpha_n} $}  \end{tabular}.
\noindent The flag variety 
  of $A_n$-type, corresponding to 
 a subset
  $
    \Delta\setminus \{\alpha_{a_1}, \cdots, \alpha_{a_k}\}$,
   parameterizes flag of linear    subspaces
  $\{V_{a_1}\leqslant \cdots \leqslant V_{a_r}\leqslant \mathbb{C}^{n+1}~|~ \dim_\mathbb{C} V_{a_j}=a_j, j=1, \cdots, r\}$
   where $[a_1,  \cdots, a_r]$ is a  subsequence  of $[1,   \cdots, n]$.  
We fix an $\alpha_k$ once and for all. Let $P\supset B$ denote the parabolic subgroup that corresponds to $\Delta_P=\Delta\setminus \{\alpha_k\}$.
Note that  the complete flag variety $F\ell_{n+1}=G/B$ and the complex Grassmannian $Gr(k, n+1)=G/P$  correspond to the
   subsequences $[1, 2, \cdots, n]$ and $[k]$ respectively, where $G=SL(n+1, \mathbb{C})$.
   The natural projection $\pi: G/B\rightarrow G/P$ is just the  forgetting map, sending a flag
       $ V_{1}\leqslant \cdots \leqslant V_{n}\leqslant \mathbb{C}^{n+1}$
       in   $F\ell_{n+1}$ to the point $V_k\leqslant \mathbb{C}^{n+1}$ in $Gr(k, n+1)$.
Furthermore, the induced map  $\pi^*: H^*(G/P)\rightarrow H^*(G/B)$ sends  a Schubert class
   $\sigma^{w}_P \in H^*(G/P)$ (where $w\in W^P$) to the Schubert class
      $\pi^*(\sigma^{w}_P)=\sigma^{w}_B \in H^*(G/B)$. Such a class $\sigma^w$ in $H^*(G/B)$ (with $w\in W^P$) is called a  \textit{Grassmannian class}.
 By abuse of notations, we skip the subscript ``$B$" and ``$P$".
  Using Theorem \ref{mainthm}, we can show Theorem \ref{mainapplication} stated in the introduction, a reformulation of which is given as follows.
  \bigskip

 \noindent\textbf{Theorem \ref{mainapplication}. }
  {\itshape   For any $u\in W^P$,   $v, w\in W$ and   $\lambda \in Q^\vee$, there exist  $v', w'\in W$ such that
     $$N_{u, v}^{w, \lambda}=N_{u, v'}^{w',\, 0}.$$
 }
To prove   the above theorem, we need the following two lemmas.

\begin{lemma}\label{lemmapositivecoe}
   Given any nonzero  $\lambda =\sum_{j=1}^na_j\alpha_j^\vee \in Q^\vee$ with $a_j\geq 0$ for all $j$,
     there exists $m\in \{1, \cdots, n\}$ such that    $\langle \alpha_m, \lambda\rangle >0$ and     $a_m>0$.
\end{lemma}
\begin{proof}
   Assume  $\langle \alpha_m, \lambda\rangle \leq 0$ for all $m$. Then $\lambda$ is a non-positive sum of fundamental coweights. As a consequence, $\lambda$ is
   a non-positive sum of simple coroots $\alpha_j^\vee$'s, by   Table 1 in section 13.2 of \cite{hum}. Thus $\lambda=0$, which contradicts the assumptions.

  Hence, there exists $m$ such that  $\langle \alpha_m, \lambda\rangle >0$.
  Consequently, we have $a_m>0$, by noting   that $\langle \alpha_m, \alpha_j^\vee\rangle $ is positive if $j=m$, or non-positive otherwise.
\end{proof}

\begin{remark}
  Lemma \ref{lemmapositivecoe} works for $\Delta$ of all types with the same proof.
\end{remark}

\begin{lemma}\label{lemmacoelargethan2}
   Let $\lambda =\sum_{j=1}^na_j\alpha_j^\vee \in Q^\vee$ with $a_j\geq 0$ for all $j$. If $\langle \alpha_m, \lambda\rangle >0$ for a unique
   $m$, then we have $\langle \alpha_m, \lambda\rangle \geq 2$.
\end{lemma}

\begin{proof}
   Let $Dyn(\tilde \Delta)$ denote the Dynkin diagram associated to a subbase $\tilde\Delta\subset \Delta$.

  Denote $\Delta':=\{\alpha_j~|~ a_j>0\}$. Clearly, $\alpha_m\in \Delta'$.    We first conclude that   $Dyn(\Delta')$     is connected. (Otherwise, we can write
    $\Delta'=\Delta_1\sqcup \Delta_2$ with     $Dyn(\Delta_1)$ being a connected  component of $Dyn(\Delta')$.
  Then    $\lambda=\lambda_1+\lambda_2$ with  $\lambda_1$ (resp. $\lambda_2$) belonging to the coroot sub-lattice of
               $\Delta_1$ (resp. $\Delta_2$). Note that $\Delta_1$ and $\Delta_2$ are orthogonal to each other.
  For each $j\in \{1, 2\}$, there exists $\alpha_{m_j}\in \Delta_j$ such that $\langle \alpha_{m_j}, \lambda_j\rangle >0$ by Lemma \ref{lemmapositivecoe}.
  This contradicts the uniqueness of $\alpha_m$.) Thus $\Delta'=\{\alpha_i, \alpha_{i+1}, \cdots, \alpha_p\}$ for some
   $1\leq i\leq m\leq p\leq n$.

When $i=p$,   the statements holds, by noting that  $\langle \alpha_m, \alpha^\vee_m\rangle=2$ and   $\lambda=a_m\alpha_m^\vee$      in this case.
When   $i<p$, we can assume $i<m$ without loss of generality. Since $0\geq \langle \alpha_i, \lambda\rangle=2a_i-a_{i+1}$, we have $a_{i+1}\geq 2a_i>a_i>0$.
Since   $0\geq \langle \alpha_{i+1}, \lambda\rangle=-a_i+2a_{i+1}-a_{i+2}$,
 we have $a_{i+2}\geq a_{i+1}+(a_{i+1}-a_i)>a_{i+1}>0$. By induction, we conclude $a_{m}>a_{m-1}>0$.
 If $m=p$, then we have  $ \langle \alpha_m, \lambda\rangle=2a_m-a_{m-1}\geq 2(a_{m-1}+1)-a_{m-1}>2$.
If $m<p$, then we can show $a_m>a_{m+1}$ with the same arguments. As a consequence,
 we have  $ \langle \alpha_m, \lambda\rangle=-a_{m-1}+2a_m-a_{m+1}\geq -a_{m-1}+(a_{m-1}+1 +a_{m+1}+1)-a_{m+1}\geq 2$.
\end{proof}

 \bigskip
\vspace{-0.2cm}

\begin{proof}[Proof of Theorem \ref{mainapplication}]
 Clearly, the statement holds  if   $N_{u, v}^{w, \lambda}$ vanishes or $\lambda=0$.

Given nonzero $\lambda =\sum_{j=1}^na_j\alpha_j^\vee \in Q^\vee$, we can assume $a_j\geq 0$ for all $j$, i.e.
 $\lambda$ is effective,
because  otherwise  $N_{u, v}^{w, \lambda}$ vanishes.  Since $\lambda\neq 0$,  there exists $m$ such that
   $\langle \alpha_m, \lambda\rangle >0$ by Lemma \ref{lemmapositivecoe}. We simply denote $\mbox{sgn}_{m}:=\mbox{sgn}_{\alpha_m}$, which is a map from $W$ to $\{0, 1\}$
   defined in the introduction (see also  section \ref{subZ2filtration}).

   If such an $m$ is not unique, then we can take any one such $m$ that is not equal to $k$.
    Since $u\in W^P$ where $\Delta_P=\Delta \setminus \{\alpha_k\}$,
   we have $\mbox{sgn}_m(u)=0$. If $\mbox{sgn}_m(v)<\mbox{sgn}_m(w)+\langle \alpha_m, \lambda \rangle$, then we have
   $N_{u, v}^{w, \lambda}=0$ by Theorem \ref{mainthm} (1);  and hence we are done. Otherwise, we have
     $\mbox{sgn}_m(v)=\mbox{sgn}_m(w)+\langle \alpha_m, \lambda \rangle=1$ and $\mbox{sgn}_m(w)=0$. By Theorem \ref{mainthm} (2),
     we have $N_{us_m, v}^{ws_m, \lambda}=N_{u, vs_m}^{ws_m, \lambda-\alpha_m^\vee}=N_{u, v}^{w, \lambda}$.

If such an $m$ is unique, then we have $\langle \alpha_m, \lambda\rangle \geq 2$ by Lemma \ref{lemmacoelargethan2}.
Thus either    $\mbox{sgn}_m(u)+\mbox{sgn}_m(v)<\mbox{sgn}_m(w)+\langle \alpha_m, \lambda \rangle$
or  $\mbox{sgn}_m(u)+\mbox{sgn}_m(v)=\mbox{sgn}_m(w)+\langle \alpha_m, \lambda \rangle$ holds.
For the former case,   $N_{u, v}^{w, \lambda}$ vanishes and then it is done. For the latter case,
  we conclude  $m=k$, $\mbox{sgn}_k(v)=1$, $\mbox{sgn}_k(w)=0$ and  $\langle \alpha_k, \lambda\rangle= 2$,
by noting that $\mbox{sgn}_j(u)=1$ if $j=k$, or $0$ otherwise.
Thus we have  $N_{u, v}^{w, \lambda}=N_{us_k, vs_k}^{w, \lambda-\alpha_k^\vee}=N_{u, vs_k}^{ws_k, \lambda-\alpha_k^\vee}$, by using
Theorem \ref{mainthm} (2) again.

Hence,   either of the followings must hold: (i)
   $N_{u, v}^{w, \lambda}=0$ (and then it is done); (ii)
       $N_{u, v}^{w, \lambda}=N_{u, vs_m}^{ws_m, \lambda'}$ with $\lambda'=\lambda-\alpha_m^\vee=\sum_{j}a_j'\alpha_j^\vee$,
       in which $|\lambda'|=|\lambda|-1$ with $a_j'=a_j-1\geq 0$ if $j=m$, or $a_j$ otherwise. Here $|\lambda|:=\sum_{j=1}^n a_j$.
Therefore, the statement holds, by using induction on $|\lambda|$.
\end{proof}

Besides Theorem \ref{mainapplication},
      we can also find other applications of Theorem \ref{mainthm}. 
\begin{example}
   Let $G/B=F\ell_4$. Take $u=v=s_2s_1s_2$, $w=s_2s_3$ and $\lambda=\alpha_1^\vee+\alpha_2^\vee$.
   (Note that neither of the Schubert classes $\sigma^u, \sigma^v$ are Grassmannian classes.)  We have
         $$N_{u, v}^{w, \lambda}=N_{u, vs_3}^{ws_3, \lambda+\alpha_3^\vee} =N_{s_2s_1s_2, s_2s_1s_2s_3}^{s_2, \alpha_1^\vee+\alpha_2^\vee+\alpha_3^\vee},$$
   in which we increase the degree   $q_\lambda$ first. Then we have
    $$N_{u, v}^{w, \lambda}=N_{s_2s_1 , s_2s_1s_2s_3}^{1, \alpha_1^\vee+\alpha_2^\vee+\alpha_3^\vee}
         =N_{s_2s_1, s_2s_1s_2}^{s_3, \alpha_1^\vee+\alpha_2^\vee}
         =N_{s_2s_1,  s_2s_1}^{s_3s_2, \alpha_1^\vee}
         =N_{s_2,  s_2}^{s_3s_2, 0}=1.$$
\end{example}
\noindent In fact, we already know all the nonzero three-pointed, genus
zero Gromov-Witten invariants for $F\ell_4$ are equal to $1$, by the multiplication table in \cite{fominGP}. Using Theorem \ref{mainthm}, we can find
their corresponding classical intersection numbers,  the most complicated case of which has been given in the above example.

  The proof of Theorem \ref{mainapplication} has also shown us how to find $v'$ and $w'$.
  Combining Theorem \ref{mainapplication} and
the  Peterson-Woodward comparison formula (Proposition \ref{comparison}),  we can obtain many nice applications, including alternative proofs of
  both  the quantum Pieri rule for all flag varieties of $A$-type given by Ciocan-Fontanine  in \cite{CFon} and the
  result that  any three-pointed genus zero Gromov-Witten invariant on a complex Grassmannian
   is a classical intersection number on a two-step flag variety of the same type, which is the central theme of \cite{BKT1} for type $A$ case by Buch, Kresch and Tamvakis.
In order to illustrate this clearly, we will show how to recover the ``quantum to classical" principle for complex Grassmannians  in the rest.

For the complex Grassmannian $X=G/P=Gr(k, n+1)$, we  note   $H_2(X, \mathbb{Z})\cong Q^\vee/Q^\vee_P\cong \mathbb{Z}$,
  so that we simply denote
 $N_{u, v}^{w, d}:=N_{u, v}^{w, \lambda_P}$ where $u, v, w\in W^P$ and $\lambda_P=d\alpha_k^\vee+Q^\vee_P$.
   Write
  $d=m_1k+r_1=m_2(n-k+1)+r_2$ where $1\leq r_1\leq k$ and $1\leq r_2\leq n-k+1$.
  Then  for $\lambda:=m_1\sum_{j=1}^{k-1}j\alpha_{j}^{\vee}+\sum_{j=1}^{r_1-1}%
j\alpha_{k-r_1+j}^{\vee}+ d\alpha_k^\vee+ m_2\sum_{j=1}^{n-k+1}j\alpha_{n+1-j}^{\vee}+\sum_{j=1}^{r_2-1}j\alpha_{k+r_2-j}^{\vee}$,
 we have     $\langle \alpha_i, \lambda\rangle=-1$ if $i\in \{k-r_1, k+r_2\}$, or $0$ otherwise.
  Thus it follows directly from the uniqueness of $\lambda_B$ 
  that  $\lambda_B=\lambda$.
 Furthermore by Proposition \ref{comparison}, we have
   $N_{u, v}^{w, d}=N_{u, v}^{\tilde w, \lambda_B} $ with
   $$\tilde w=w\omega_P\omega_{P'}=wu_{k-r_1}^{(k-1)}u_{k-r_1}^{(k-2)}\cdots u_{k-r_1}^{(k-r_1)}  v_{n+1-k-r_2}^{(n-r_2+1)}v_{n+1-k-r_2}^{(n-r_2+2)}\cdots v_{n+1-k-r_2}^{(n)}$$ (see e.g. Lemma 3.6 of \cite{leungli33} for the way of obtaining $\omega_P\omega_{P'}$).
  Here  for any $1\leq i\leq m$, we  denote $u_i^{(m)}:=s_{m-i+1}\cdots s_{m-1}s_m$ and $v^{(m)}_i=\big(u^{(m)}_i\big)^{-1}=s_ms_{m-1}\cdots s_{m-i+1}$; in addition,
 we denote  $u_0^{(m)}=v_0^{(m)}=\mbox{id}$.

   In particular, if    $1\leq d \leq \min\{k, n+1-k\}$, then we have $d=r_1=r_2$ and  $\Delta_{P'}=\Delta_P\setminus \{\alpha_{k-d}, \alpha_{k+d}\}$.
  Furthermore in this case, we go through the proof of Theorem \ref{mainapplication}
    for the above special $\lambda_B$, by reducing it to the zero coroot according to the ordering
 $\big((\alpha_k^\vee, \alpha_{k-1}^\vee, \cdots, \alpha_{k-d+1}^\vee),  (\alpha_{k+1}^\vee, \alpha_{k+2}^\vee, \cdots, \alpha_{k+d-1}^\vee), (\alpha_{k}^\vee, \cdots,
  \alpha_{k-d+2}^\vee),$ $(\alpha_{k+1}^\vee, \cdots, \alpha_{k+d-2}^\vee), \cdots, (\alpha_{k}^\vee, \alpha_{k-1}^\vee), (\alpha_{k+1}^\vee), \alpha_k^\vee\big)$.
  Correspondingly, we denote
   $$x:=v_d^{(k)}u_{d-1}^{k+d-1}v_{d-1}^{(k)}u_{d-2}^{(k+d-2)}\cdots v_{2}^{(k)}u_1^{(k+1)}s_k.$$
\noindent (Note $\ell(x)=d^2$.) As a direct  consequence, we have

 \begin{cor}\label{corofappforGrass}
    For any $u, v, w\in W^P$ and $d\in \mathbb{Z}$ with  $1\leq d \leq \min\{k, n+1-k\}$, we have
      $N_{u, v}^{w, d}=N_{u, vx}^{\tilde w x, 0}$,
      provided that $\ell(vx)=\ell(v)-\ell(x)$ and $\ell(\tilde wx)=\ell(\tilde w)+\ell(x)$, and zero otherwise.
  \end{cor}

Let $\bar P\supset B$ denote the parabolic subgroup that corresponds to the subset $\Delta\setminus \{\alpha_{k-d}, \alpha_{k+d}\}$.
 That is, $G/\bar P=F\ell_{k-d, k+d; n+1}=\{V\leqslant V'\leqslant \mathbb{C}^{n+1}~|~ \dim V=k-d, \dim V'=k+d\}$ is a two-step flag variety. We can reprove the next result of Buch, Kresch and Tamvakis.

\begin{prop}[Corollary 1 of \cite{BKT1}]\label{propQtoCforGrass}
 For any Schubert classes $\sigma^{u}, \sigma^{v}, \sigma^{w}$   in $H^*(Gr(k, n+1), \mathbb{Z})$ and any $d\geq 1$,
     the Gromov-Witten invariant $N_{u, v}^{w, d}$ coincides with
      the classical intersection number $N_{ux, vx}^{\tilde w, 0}$ for $\sigma^{ux}\cup \sigma^{vx}$ in $H^*(F\ell_{k-d, k+d; n+1}, \mathbb{Z})$,
       provided that  $d\leq \min\{k, n+1-k\}$, $\ell(ux)=\ell(u)-\ell(x)$,    $\ell(vx)=\ell(v)-\ell(x)$ and $\tilde w\in W^{\bar P}$, and vanishes otherwise.
\end{prop}

To show the above proposition, we need the next two lemmas.
\begin{lemma}\label{lemmavanifordegreelar}
 For any Schubert classes $\sigma^{u}, \sigma^{v}, \sigma^{w}$   in $H^*(Gr(k, n+1), \mathbb{Z})$,
     the Gromov-Witten invariant $N_{u, v}^{w, d}$ vanishes unless $0\leq d\leq \min\{k, n+1-k\}$.
\end{lemma}
 \begin{proof}
Note $N_{u, v}^{w, d}=N_{u, v}^{\tilde w, \lambda_B}$. 
    If $k=1$ (resp. $n$),
  then     $\langle \alpha_k, \lambda_B\rangle=d+m_2+1$ (resp. $d+m_1+1$) is larger than $2$, whenever $d>1=\min\{k, n+1-k\}$.
    Thus  we have $N_{u, v}^{\tilde w, \lambda_B}=0$ by Theorem \ref{mainthm} (1).
  If  $2\leq k\leq n-1$, then we have  $\langle \alpha_k, \lambda_B\rangle=m_1+m_2+2$.
  By Theorem \ref{mainthm} (1) again, we   have $N_{u, v}^{\tilde w, \lambda_B} =0$ unless
     $m_1=m_2=0$, in which case we still have   $d=r_1=r_2\leq   \min\{k, n+1-k\}$.
  \end{proof}

\begin{lemma}\label{lemmafortwostepele}
  For any $v\in W^P$, we have $vx\in W^{\bar P}$ if $\ell(vx)=\ell(v)-\ell(x)$.
\end{lemma}

 \begin{proof}
    Since  $\ell(vx)=\ell(v)-\ell(x)$, we have $\ell(vx)=\ell(vxs_k)-1$, so that $vx(\alpha_k)\in R^+$.
    For any $j\in \{1, \cdots, k-d-1, k+d+1, \cdots, n\}$, we have $vx(\alpha_j)=v(\alpha_j)\in R^+$.
    For   $j\in \{k-d+1, \cdots, k-1\}$, we have
    $vx(\alpha_j)=vv_d^{(k)}u_{d-1}^{k+d-1} \cdots v_{k-j+1}^{(k)}u_{k-j}^{(2k-j)}v_{k-j}^{(k)}(\alpha_j)$
     $=vv_d^{(k)}u_{d-1}^{k+d-1} \cdots v_{k-j+2}^{(k)}u_{k-j+1}^{(2k-j+1)}(\alpha_{k+1}) =v(\alpha_{k+d})\in R^+.$
    Similarly for  $j\in \{k+1, \cdots, k+d-1\}$, we have
    $vx(\alpha_j)=v(\alpha_{j-d})\in R^+$.
 Hence, we have $vx\in W^{\bar P}$.
 \end{proof}

 \begin{remark}\label{remdescuvw}  The Weyl group $W$ for $G=SL(n+1, \mathbb{C})$ is canonically isomorphic to the permutation group $S_{n+1}$ by mapping each
  simple reflection $s_i\in W$ to the transposition $(i, i+1)\in S_{n+1}$.
   In ``one-line" notation,  each permutation $t\in W=S_{n+1}$ is  written as $(t(1), \cdots, t(n+1))$. In particular, Grassmannian permutations
     $t\in W^P$ for $G/P=Gr(k, n+1)$ are precisely the permutations with (at most) a single descent occurring at $k${\upshape  th} position (i.e., $t(k)>t(k+1)$).
    Permutations
     $t\in W^{\bar P}$ for $G/\bar P=F\ell_{k-d, k+d; n+1}$ are precisely the permutations with (at most) two   descents occurring at $(k-d)${\upshape  th} and $(k+d)${\upshape  th} position.
           With this characterization,  $vx\in W^{\bar P}$ is the element obtained from $v\in W^P$
     by sorting the values $\{v(k-d+1), \cdots, v(k+d)\}$  to be in increasing
order, which coincides with the descriptions in section 2.2 of \cite{BKT1}. (
          Indeed,  we note that $s_j(i)=i$ for any $j\in \{k-d+1,   \cdots, k+d-1\}$ and any
      $i\in \{1, \cdots, k-d, k+d+1, \cdots, n+1\}$.
   Thus we have $vx(i)=v(i)$ for any such $i$ and consequently
     the set  $\{v(k-d+1), \cdots, v(k+d)\}$ coincides with the set $\{vx(k-d+1), \cdots, vx(k+d)\}$.)
    Similarly, we can show that  $\tilde w$ is the permutation $(w(d+1), \cdots, w(k), \tilde w(k-d+1), \cdots, \tilde w(k+d),w(k+1),\cdots, w(n-d+1))$, in which $(\tilde w(k-d+1), \cdots, \tilde w(k+d))$ is  obtained from $w\in W^P$
     by sorting the values $\{w(1),\cdots, w(d), w(n-d+2), \cdots, w(n+1)\}$  to be in increasing
order.
 \end{remark}

\begin{proof}[Proof of Proposition \ref{propQtoCforGrass}]
   It follows from Lemma \ref{lemmavanifordegreelar} and Corollary \ref{corofappforGrass} that
     $N_{u, v}^{\tilde w, d}=N_{u, vx}^{\tilde wx ,0}$ if $1\leq d \leq \min\{k, n+1-k\}$, $\ell(vx)=\ell(v)-\ell(x)$ and $\ell(\tilde wx)=\ell(\tilde w)+\ell(x)$,
      or $0$ otherwise.
  When all of these hold, we have $vx\in W^{\bar P}$ by Lemma \ref{lemmafortwostepele} and  note that $x$ is in the Weyl subgroup generated by
    $\{s_{k-d+1}, \cdots, s_{k+d-1}\}$.
  In particular for any $j\in \{k-d+1, \cdots, k+d-1\}$, we have $\mbox{sgn}_j(vx)=0$. Thus
   by Theorem \ref{mainthm}, we have $N_{u, vx}^{\tilde wx, \, 0}=N_{ux^{-1}, vx}^{\tilde w, \, 0}$ provided that $\ell(ux^{-1})=\ell(u)-\ell(x^{-1})$, and zero otherwise. This assumption implies that $ux\in W^{\bar P}$, because of the observation that   $x=x^{-1}$.
   As a consequence,  $N_{ux, vx}^{\tilde w, \, 0}=0$ unless $\tilde w\in W^{\bar P}$, for which the assumption ``$\ell(\tilde wx)=\ell(\tilde w)+\ell(x)$" holds automatically.
\end{proof}

\section{Acknowledgements}
 The authors thank Ionut Ciocan-Fontanine, Bumsig Kim, Leonardo Constantin Mihalcea and Harry Tamvakis
   for useful discussions. We also   thank the referees for valuable suggestions.
 The work described in this paper was substantially    supported by a grant from the Research Grants Council of the Hong Kong Special Administrative Region, China (Project No. CUHK401908), and  partially by   the KRF grant 2007-341-C00006.

\end{document}